\documentclass[reqno,12pt]{amsart}

\usepackage{enumerate}
\usepackage[margin=1in]{geometry}
\usepackage{ifpdf}
\usepackage{amsmath}
\usepackage{amsfonts}
\usepackage{amssymb}
\usepackage{amsthm}
\usepackage[hyperfootnotes=false]{hyperref}
\usepackage{amsrefs}
\usepackage{graphicx}

\newcommand{\ignore}[1]{}

\newcommand{\abs}[1]{\left\lvert {#1} \right\rvert}
\newcommand{\norm}[1]{\left\lVert {#1} \right\rVert}
\newcommand{\sabs}[1]{\lvert {#1} \rvert}
\newcommand{\snorm}[1]{\lVert {#1} \rVert}

\newcommand{\C}{{\mathbb{C}}}
\newcommand{\R}{{\mathbb{R}}}

\newcommand{\Q}{{\mathbb{Q}}}

\newcommand{\bN}{{\mathbb{N}}}

\newcommand{\sN}{{\mathcal{N}}}

\newcommand{\sP}{{\mathcal{P}}}

\newcommand{\sZ}{{\mathcal{Z}}}

\newcommand{\rank}{\operatorname{rank}}

\newcommand{\bfe}{\mathbf{e}}

\newtheorem{thm}{Theorem}[section]

\newtheorem{prop}[thm]{Proposition}

\newtheorem{cor}[thm]{Corollary}

\newtheorem{lemma}[thm]{Lemma}

\theoremstyle{definition}

\theoremstyle{remark}
\newtheorem{remark}[thm]{Remark}

\author{Jennifer Halfpap}
\thanks{The first author was supported in part by NSF grant DMS 1200815.}
\address{Department of Mathematical Sciences, University of Montana, Missoula, MT 59812, USA}
\email{halfpap@mso.umt.edu}

\author{Ji\v{r}\'{\i} Lebl}
\thanks{The second author was supported in part by NSF grant DMS 0900885.}
\address{Department of Mathematics, University of Wisconsin,
Madison, WI 53706, USA}
\email{lebl@math.wisc.edu}

\date{April 19, 2013}

\ifpdf
\hypersetup{
  pdftitle={Signature pairs of positive polynomials},
  pdfauthor={Jennifer Halfpap and Jiri Lebl}
}
\fi

\title{Signature pairs of positive polynomials}

\begin{document}


\begin{abstract}
A well-known theorem of Quillen says that if $r(z,\bar{z})$ is a
bihomogeneous polynomial on $\C^n$ positive on the sphere,
then there exists $d$ such that
$r(z,\bar{z})\snorm{z}^{2d}$ is a squared norm.
We obtain effective bounds relating this $d$ to the signature of $r$.
We obtain the sharp bound for $d=1$, and for $d > 1$ we obtain
a bound that is of the correct order as a function of $d$
for fixed $n$. The current work adds to an extensive literature on positivity classes for real polynomials.
The classes $\Psi_d$ of polynomials for which
$r(z,\bar{z})\snorm{z}^{2d}$ is a squared norm
interpolate between polynomials positive on the sphere and
those that are Hermitian sums of squares.
\end{abstract}

\maketitle

\section{Introduction}

Let $r(z,\bar{z})$ be a real polynomial
on $\C^n$.  A basic question one can ask is whether $r(z,\bar{z}) \geq 0$.
One way to show that a polynomial is
nonnegative is to write it as a sum of Hermitian squares
\begin{equation} \label{eq:sos}
\sum_{j=1}^N \abs{f_j(z)}^2
\end{equation}
for holomorphic polynomials $f_j$, i.e., as the squared norm
$\snorm{f(z)}^2$
of a holomorphic mapping $f$.
There exist, however, nonnegative polynomials that cannot be written as a squared norm; to construct an easy example, take $r$ nonnegative but with zero set a real hypersurface. For a much more subtle example, consider Example VI.3.6 in \cite{DAngeloCarus}:
\begin{equation}
r(z,\bar{z})=(\abs{z_1 z_2}^2-\abs{z_3}^4)^2+\abs{z_1}^8.
\end{equation}
This polynomial is non-negative, its zero set is a complex line, and yet it cannot even be written as a quotient of squared norms.

Thus the condition that a real polynomial is a squared norm is too restrictive, and one is motivated to formulate other less restrictive positivity conditions.
See
\cites{DAngeloCarus,Quillen,CD,DAngelo:hilbert,DAngeloVarolin} and the references within.
A theorem of Quillen \cite{Quillen}, proved independently by Catlin and D'Angelo
\cite{CD}, states that
if a bihomogeneous polynomial
\begin{equation}
r(z,\bar{z})=\sum_{\sabs{\alpha}=\sabs{\beta}=m}
c_{\alpha \beta}z^\alpha \bar{z}^\beta
\end{equation}
is positive on the unit sphere, then there exists an integer $d$ such that
$r(z,\bar{z}) {(\norm{z}^2)}^d = r(z,\bar{z}) \norm{z}^{2d}$
is a squared norm, and hence $r$ is a quotient of squared norms.
Thus one obtains a Hermitian analogue of Hilbert's 17th problem.

With this motivation, we define a set of positivity classes $\Psi_d$ of bihomogeneous polynomials
by
\begin{align}
& \Psi_d = \{ r : r(z,\bar{z}) \norm{z}^{2d} \text{ is a Hermitian sum of
squares} \} ,
\\
& \Psi_\infty = \bigcup_{d=0}^\infty \Psi_d .
\end{align}
$\Psi_0$ consists of the squared norms themselves, and, by the theorem
mentioned above, $\Psi_\infty$
contains the polynomials positive on the sphere. It is not difficult (Proposition~\ref{prop:qsdistinct}) to construct polynomials that show
\begin{equation}
\Psi_0 \subsetneq
\Psi_1 \subsetneq
\Psi_2 \subsetneq
\Psi_3 \subsetneq \dots
.
\end{equation}

Every real polynomial $r$ has a holomorphic decomposition
\begin{equation}\label{eq: holo decomp}
r(z,\bar{z}) = \sum_{j=1}^{N_+} \abs{f_j(z)}^2 - \sum_{j=1}^{N_-}
\abs{g_j(z)}^2
\end{equation}
for holomorphic polynomials $f_j$, $g_j$.
When $N_+$ and $N_-$ are minimal (which occurs when
$f_1,\ldots,f_{N_+},g_1,\ldots,g_{N_-}$ are linearly independent),
we say that $r$ has signature pair $(N_+,N_-)$ and rank
$N_++N_-$.  While $f$ and $g$ are not unique, the signature pair $(N_+,N_-)$
is.

We will be particularly concerned with $\Psi_1$.  This class is
connected to the study of proper holomorphic mappings between balls in
complex Euclidean spaces of different dimensions.
For example, if $f \colon \C^n \to \C^N$ is a polynomial that
takes the unit ball to the unit ball properly,
then $\snorm{f(z)}^2-1 = p(z,\bar{z}) (\snorm{z}^2-1)$.  In particular,
if $f$ is of degree $d$ and $f_d$ is the degree $d$ part
of $f$, then
$\snorm{f_d(z)}^2 = p_{d-1}(z,\bar{z}) \snorm{z}^2$.
Polynomials in $\Psi_1$ also arise when studying the second fundamental
form of more general mappings between balls.  See the recent work
by Ebenfelt~\cite{Ebenfelt:partrig} and
the references within.  For example, by proving
that $p(z,\bar{z}) \snorm{z}^2$ must be of rank at least $n$,
Huang~\cite{huang:lin} proved that all proper mappings between
balls that are sufficiently smooth on the boundary
are equivalent to the linear embeddings if $N < 2n-1$.

Our main result for the positivity class $\Psi_1$ is the following.
\begin{thm} \label{thm:thmd1}
Let $r(z,\bar{z})$ be a real polynomial on $\C^n$, $n \geq 2$, and suppose
that $r(z,\bar{z}) \norm{z}^2$ is a squared norm.  Let $(N_+,N_-)$ be the
signature pair of $r$.  Then
\begin{enumerate}[(i)]
\item
\begin{equation} \label{eq:theboundd1}
\frac{N_-}{N_+} < n-1.
\end{equation}
\item The above inequality is sharp, i.e., for every $\varepsilon >0$ there exists $r$ with
$\frac{N_-}{N_+} \geq n-1-\varepsilon$.
\end{enumerate}
\end{thm}

\begin{remark}
The case for $n=1$ is trivial; $\norm{z}^2 =
\abs{z}^2$ and so $r(z,\bar{z})\abs{z}^2$ has the same signature
as $r$. Therefore, if $r(z,\bar{z})\abs{z}^2$ is a squared norm,
then $r$ is a squared norm and $N_- = 0$.
\end{remark}

When $d > 1$, the combinatorics becomes more involved.  We obtain
the following bound.
\begin{thm} \label{thm:thmalld}
Let $r(z,\bar{z})$ be a real polynomial on $\C^n$, $n \geq 2$, $d \geq 1$,
and
suppose that $r(z,\bar{z}) \norm{z}^{2d}$ is a squared norm.
Let $(N_+,N_-)$
be the signature pair of $r$.  Then
\begin{enumerate}[(i)]
\item
\begin{equation}
\frac{N_-}{N_+} \leq \binom{n-1+d}{d} -1.
\end{equation}
\item For each fixed $n$, there exists a positive constant $C_n$ such that
for each $d$ there is a polynomial $r \in \Psi_d$ with
$\frac{N_-}{N_+} \geq C_n d^{n-1}$.
\end{enumerate}
\end{thm}

Since $\binom{n-1+d}{d}$ is a polynomial in $d$ of degree $n-1$, the second item says that the bound we obtain is of the
correct order, although we do not believe it to be sharp for all $n$ (it is
sharp when $n=2$).

For bihomogeneous polynomials we obtain bounds for the ratios of positive
and negative eigenvalues for the classes $\Psi_d$.  A very interesting
problem is to find the smallest $d$ so that
a positive polynomial is in $\Psi_d$; see
the work of To and Yeung~\cite{ToYeung}.
An upper bound must involve
the magnitude of the coefficients.
To see this, consider an example from \cite{DAngeloVarolin}:
\begin{equation}
{(\abs{z}^2+\abs{w}^2)}^4-\lambda\abs{zw}^4 .
\end{equation}
As $\lambda \to 16$, one requires larger and larger $d$.
On the other hand, our results give an effective lower bound on $d$
given the numbers $N_-$ and $N_+$.

We also address the analogous question for real polynomials, i.e., what can we say about a polynomial
$p \in \R[x_1,\ldots,x_n]$ if it is known that $(x_1+\cdots+x_n)^d p(x)$ has
non-negative coefficients?  P\'olya proved in~\cite{Polya} that for each
$p$ positive on the positive quadrant, there exists a $d$
such that $(x_1+\cdots+x_n)^d p(x)$ has only positive coefficients.
Recent work (for example~\cite{PowersReznick}) focuses on finding
an upper bound on $d$
given information about $p$.
Our work can be thought of as finding
lower bounds on $d$ given the signature of $p$ in a somewhat more general setting.

When we complexify a real polynomial we obtain a Hermitian symmetric
polynomial, i.e., one satisfying $r(z,\bar{w}) = \overline{r(w,\bar{z})}$.  Hermitian
symmetric polynomials arise naturally in complex geometry, in particular,
degree $d$ Hermitian symmetric polynomials arise as globalizable metrics on
the $d$th power of the universal bundle over the complex projective space;
see~\cite{DAngeloVarolin}.

Questions about multiples of $\snorm{z}^{2d}$ also arise in several contexts. As mentioned above,
Huang~\cite{huang:lin} proved that
$p(z,\bar{z}) \snorm{z}^2$ must have rank at least $n$.
Generalizing this result, in \cite{DL:pfi} it was shown that the rank of
$p(z,\bar{z}) \snorm{z}^{2d}$ is bigger than or equal to
the rank of $\snorm{z}^{2d}$.  A theorem of
Pfister says the if $p \geq 0 $ for a polynomial $p$ of $n$ real variables, there exists a polynomial $q$ such that $q^2p$ is a sum of at most $2^n$ squared polynomials.  Thus  \cite{DL:pfi} shows that Pfister's theorem fails in the Hermitian context.

Finally, ratios of the sort considered have been
studied recently by Grundmeier~\cite{Grundmeier} in the context of
group invariant hyperquadric CR mappings.  In particular, Grundmeier
studied the canonically defined group-invariant mappings
from the ball to the hyperquadric.  This problem can be seen as studying
the proportions of positive and negative eigenvalues of group-invariant
polynomials of the form
$p(z,\bar{z}) (\snorm{z}^2-1)$.

The authors would like to acknowledge Peter Ebenfelt, whose question
led to this research.  We would also like to express our gratitude to John P. D'Angelo for many fruitful conversations.  Finally, we thank Iris Lee for her sense of humor.


\section{Preliminaries}

Let $r(z,\bar{z})$ be a real-valued polynomial on $\C^n$.
We use a linear algebra setting.
Suppose $\deg r \leq 2D$.  Let $\sZ =
(1,z_1,\ldots,z_n,z_1^2,z_1z_2,\ldots,z_n^D)^t$ be the vector of
all monomials up to degree $D$, with $\sZ^*$ its conjugate
transpose.  Then
$r(z,\bar{z}) = \sZ^* C \sZ$
where $C$ is a constant Hermitian matrix.  The rank of $r$ is the rank of
$C$, and the signature of $r$ is $(N_+,N_-)$ if and only if
$C$ has $N_+$ positive and $N_-$ negative eigenvalues.  Therefore, when we apply linear algebra terminology to $r$ we are referring to properties of the matrix $C$.

When $r$ is
diagonal, the $f_j$ and $g_j$ appearing in the holomorphic
decomposition \eqref{eq: holo decomp} are monomials.
In this case, questions about $r(z,\bar{z})$ and $r(z,\bar{z}) \norm{z}^{2d}$ for $z \in \C^n$ can be reformulated as questions about polynomials on
$\{\,x\in\R^n:x_k\geq 0\,\}$.  Indeed,
if in
\eqref{eq: holo decomp}, each $f_j$, $g_j$ is a monomial $c_\alpha z^\alpha$ for some multi-index
$\alpha$, then
\begin{equation}
\abs{c_\alpha z^\alpha}^2 = \abs{c_\alpha}^2 \prod_{k=1}^n
(\abs{z_k}^2)^{\alpha_k}.
\end{equation}
If $m_j \colon \R^n \to \mathbb{R}$ is given by $m_j(x):=\abs{c_\alpha}^2
x^\alpha$, then $\abs{f_j(z)}^2 = m_j(\abs{z_1}^2, \ldots,
\abs{z_n}^2)$.  Thus we can study $r$ by studying an associated
real polynomial $p$ on $\{\,x\in\R^n: x_k \geq 0 \,\}$ with $N_+$
positive and $N_-$ negative coefficients.  Observe that
$\norm{z}^2$ is itself a diagonal polynomial and is associated
with $\ell(x):=\sum_{k=1}^n x_k$.

One can therefore formulate the associated problem for real polynomials.
We consider real polynomials $p(x)$ for which
$p(x) \ell(x)^d$
has only nonnegative coefficients.
Such polynomials are nonnegative on
$\{\,x\in\R^n: x_k \geq 0 \,\}$.
Since it is not hard to see how to go from a real polynomial $p(x)$ on $\R^n$
to its Hermitian analogue $r(z,\bar{z})$ on $\C^n$, if we construct
a $p(x)$ with signature $(N_+,N_-)$, we automatically also
construct an $r(z,\bar{z})$ with the same signature $(N_+,N_-)$.


\section{Diagonal case for $d=1$}

In this section we focus on the diagonal case.  The combinatorics
in this special case gives insight into the general case, and furthermore,
we obtain somewhat stronger results.

The polynomials we construct to establish the sense in which our bounds are sharp are all diagonal.
Thus the second part of
Theorem~\ref{thm:thmd1} follows immediately from the last part of the next theorem.

\begin{thm} \label{thm:monthm1}
Suppose $p$ is a polynomial on $\R^n$, $n \geq 2$, and set
$\ell(x):=\sum_{j=1}^n x_j$. Suppose $S(x):=p(x)\ell(x)$ has only
nonnegative coefficients. Let $N_+$ denote the number of 
monomials in $p$ with positive coefficients, and let $N_-$ denote
the number of monomials in $p$ with negative coefficients.
\begin{enumerate}[(i)]
\item If $N_- > 0$, then $N_+ \geq n$.

\item $\frac{N_-}{N_+} < n-1$.

\item For every $\varepsilon >0$, there exists $p$ with
$\frac{N_-}{N_+} > n-1 -\varepsilon$.
\end{enumerate}
\end{thm}

Before proving the theorem, we describe a useful visualization for our
constructions.  Consider homogeneous polynomials in $n=3$
variables.  To avoid subscripts $x,y,z$.
Thus we consider polynomials
$p(x,y,z)$ such that $S(x,y,z)=p(x,y,z) (x+y+z)$
has only nonnegative coefficients.
In Figure~\ref{fig:newton1neg}, we show a diagram for
the polynomial $p(x,y,z) = x^2+y^2+xz - xy$.  We arrange the monomials in a
lattice and mark positive coefficients by a $P$ in a thick circle
and negative coefficients by an $N$ in a thin circle.
In this first diagram, we indicate which monomial each circle
represents, though we refrain from doing so for larger diagrams.
Zero coefficients are marked with dotted circle and do not really come into
play.  We also mark by gray triangles the monomials appearing in the product
$S(x,y,z)=p(x,y,z)(x+y+z)$.  The vertices of each triangle point to monomials of $p$
that contribute to that term of $S$.
We ignore the
magnitude of the coefficients; we are only interested in their signs.
If a term in the product $S$ receives contributions from both
positive and negative terms in $p$, we can increase the positive coefficients so that the sum of the
positive contributions is bigger than the sum of the negative contributions,
thus ensuring that $S$ has only positive coefficients.  For $S$ to have only
positive coefficients, each nonzero term in $S$ must get at least one positive
contribution, and hence each triangle must have one vertex pointing
to a $P$ in the diagram of $p$.
We do not show triangles that receive no contribution from a term in $p$.

\begin{figure}[h!t]
\begin{center}
\includegraphics{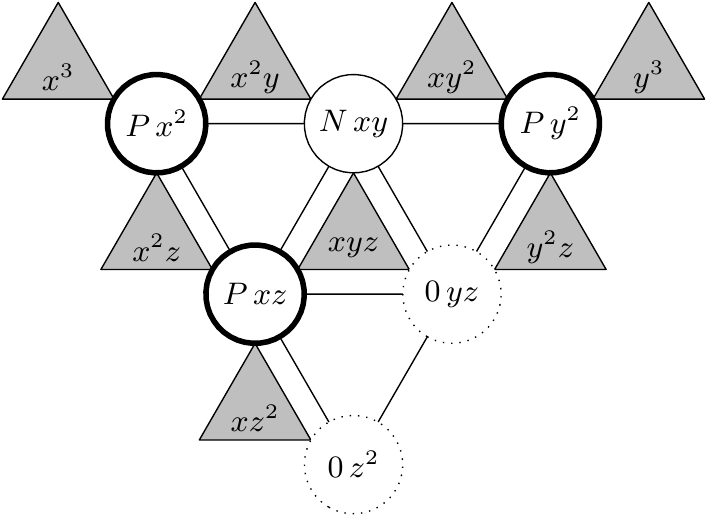}
\caption{A diagram for a second-degree $p(x,y,z)$ with one negative term.}\label{fig:newton1neg}
\end{center}
\end{figure}

Figure~\ref{fig:newton6neg} shows the diagram for a polynomial $p$ with 6 negative coefficients.
\begin{figure}[h!t]
\begin{center}
\includegraphics{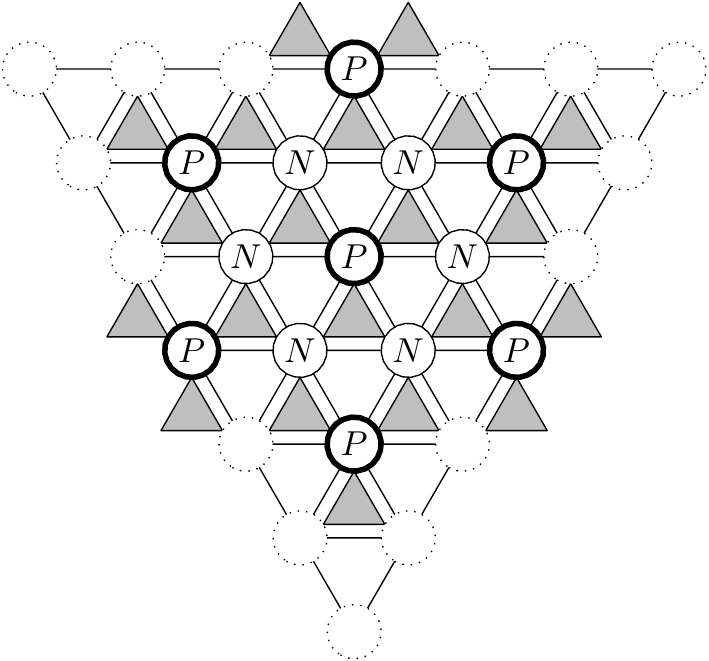}
\caption{Diagram for an example with 6 negative
terms.}\label{fig:newton6neg}
\end{center}
\end{figure}
The key point is that each gray triangle has
at least one vertex pointing to a $P$ in the diagram.  It is not hard
to argue that, if we have 6 negative terms, we must have at least 7 positive
terms. Thus
this figure is in some sense optimal.  An
explicit polynomial having the diagram of Figure~\ref{fig:newton6neg} is
\begin{multline}
p(x,y,z) =
2xyz^4+2x^3z^3+2y^3z^3+2x^2y^2z^2+2x^4yz+2xy^4z+2x^3y^3\\
-x^2yz^3-xy^2z^3-x^3yz^2-xy^3z^2-x^3y^2z-x^2y^3z .
\end{multline}
With 7 positive and only 6 negative coefficients, $N_-/N_+$ is far
from the predicted bound of 2.  Furthermore, this polynomial is already of
degree 6.  To obtain polynomials with ratio close to the bound, we must take
the degree to be
very large, and it is impractical to give diagrams for specific examples.
However, the pattern in Figure~\ref{fig:newton6neg} can be extended to obtain our ``sharp" examples;
the idea is to make
the interior of the diagram as in Figure~\ref{fig:newtontile}.  (We omit the triangles.)

\begin{figure}[h!t]
\begin{center}
\includegraphics{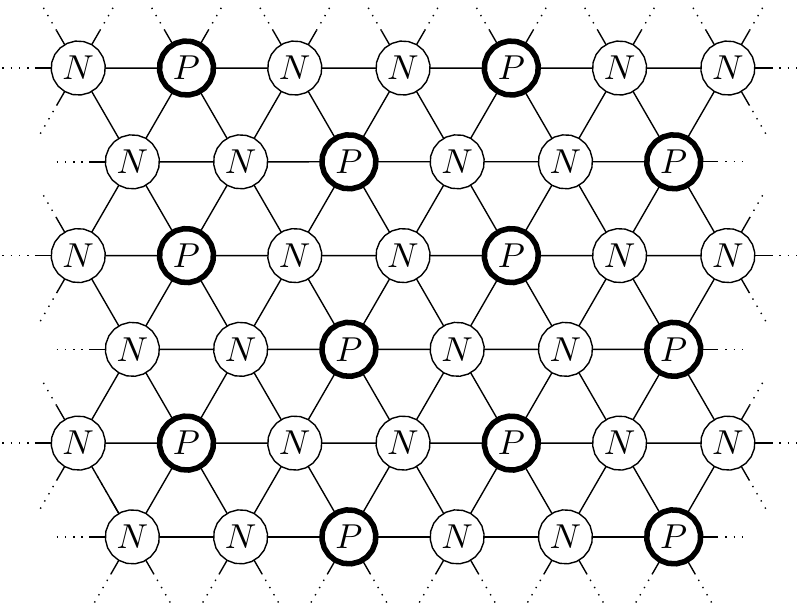}
\caption{Interior of an optimal diagram for $n=3$.}\label{fig:newtontile}
\end{center}
\end{figure}

In order to make the diagram correspond to a polynomial, we
 must make this pattern part of a finite diagram.  We will show that if we take all terms on the boundary to be positive, no negative terms will appear in $S$. Because of these boundary terms, we will have
slightly more than one positive term for every two negative terms.

\begin{proof}[Proof of Theorem~\ref{thm:monthm1}]
For any polynomial $p$ in $n$ variables,
write $p=\sum^D_{j=0} p_j$ where each $p_j$ is homogeneous of
degree $j$. Since we obtain $S$ by multiplying $p$ by
a homogeneous polynomial of degree one, if
$S=\sum_{j=0}^D S_{j}$ with $S_{j}$ homogeneous of degree
$j+1$, $S_{j}$ is simply $p_j \ell$.
One shows easily that, for each statement above, if it holds for each
$p_j$, it holds for $p$.
Thus for the remainder of the proof, we assume all polynomials
are homogeneous and that like terms have been
collected, so that a polynomial is a sum of
distinct monomials.

\paragraph{Proof of (i).} Suppose that in $p$, the coefficient
of $x^{\beta}$ is negative.  This
coefficient contributes to the coefficients of $n$ distinct
terms in $S$ associated with multi-indices $\beta+\bfe^k$, $1 \leq
k \leq n$, where $\bfe^k$ is the vector with 1 in the $k$th position and
zero elsewhere.
For each $k$, there must be a multi-index
$\alpha(k)$ associated with a positive coefficient in $p$ for which
$\alpha(k) + \bfe^j=\beta+ \bfe^k$ for some $j$.  We claim that if $k_1 \ne k_2$, $\alpha(k_1)$ cannot equal $\alpha(k_2)$. Suppose, on the contrary, that there is a single multi-index $\alpha$ different from
$\beta$ and integers $1 \leq j_1, j_2, k_1, k_2\leq n$ with $j_1 \ne
j_2$ and $k_1 \ne k_2$ such that
$\alpha+\bfe^{j_1}=\beta+\bfe^{k_1}$ and $\alpha +
\bfe^{j_2}=\beta+\bfe^{k_2}$. Then
$\bfe^{k_1}-\bfe^{j_1}-\bfe^{k_2}+\bfe^{j_2}=0$.  $\bfe^{j_i}\ne
\bfe^{k_i}$ since $\alpha \ne \beta$, so it must be that
$\bfe^{k_1}=\bfe^{k_2}$ and $\bfe^{j_1}=\bfe^{j_2}$.  This is a
contradiction. We conclude that there are indeed at least $n$
distinct multi-indices $\alpha$ for which the coefficient of $x^\alpha$ in $p$ is positive.

\paragraph{Proof of (ii).}
Let $\sN$ be the set of multi-indices $\alpha$ for which the
coefficient of $x^\alpha$ in $p$ is negative, and let $\sP$ be the
set of all multi-indices for which the coefficient of $x^\alpha$
is positive.  Since $|\sN|=N_-$ and $|\sP|=N_+$, the result will
follow if whenever $\sN$ is nonempty, there exists a function
$f \colon \sN \to \sP$ for which $f^{-1}(\{\beta\})$ has at
most $n-1$ elements for each $\beta \in \sP$.

Consider $\alpha \in \sN$.  The negative coefficient $c_\alpha$ in
$p$ contributes to $n$ terms in $S$, among them the one associated
with $\alpha':=\alpha+\bfe^n$.  The other
multi-indices from $p$ that contribute to this term are
$\alpha'-\bfe^j$ for $1 \leq j \leq n-1$.  In order for the
coefficient of $x^{\alpha'}$ in $S$ to be non-negative, there must
exist $j$ for which $\alpha' -\bfe^{j} \in \sP$. We choose
$j_0$ to be minimal with this property and set
$f(\alpha):=\alpha'-\bfe^{j_0}=\alpha+\bfe^n-\bfe^{j_0}$.

Fix $\beta \in \sP$ and consider $f^{-1}(\{\beta\})$.  If $\alpha$
is such a pre-image, $\beta = f(\alpha)=\alpha+\bfe^n-\bfe^{j} $
for some $j$ between $1$ and $n-1$.  Thus $\alpha$ must be of the
form $\beta - \bfe^n + \bfe^{j}$ for $1 \leq j \leq n-1$, i.e.,
$|f^{-1}(\{\beta\})|\leq n-1$.

Since $p(x) \ell(x)$ has nonnegative coefficients, if we look at the least monomial $x^{\beta_0}$ in $p$ (according to our monomial order)
with nonzero coefficient, we find that
it must be positive because it is the only coefficient of $p$ contributing to the coefficient of $x_1 x^{\beta_0}$ in $S$.
Furthermore, $f^{-1}(\{\beta_0\})$ is empty; such a pre-image would be of the form $\beta_0+\bfe^j-\bfe^n$. Since $x^{\beta_0+\bfe^j-\bfe^n}<x^{\beta_0}$ in the monomial order, $x^{\beta_0+\bfe^j-\bfe^n}$ does not appear in $p$ with non-zero coefficient. This proves that the inequality is in fact strict.

\paragraph{Proof of (iii).}
We construct a family $\{p_D:D\in \bN\}$
of polynomials with $p_D$ homogeneous of degree $D$ such that
$N_-(p_D)/N_+(p_D) \to n-1$ as $D \to \infty$.

For each multi-index $\alpha$, set
\begin{equation}\label{eq: def of coeff.}
\gamma(\alpha):=\begin{cases}n-1 & \text{if $\alpha_k = 0$ for
some $k$}\\
n-1 & \text{if $\alpha_k \geq 1$ for all $k$ and $\sum_{k=1}^{n-1}
k\alpha_k -D \equiv 0 \mod n$}\\
-1 & \text{otherwise}.
\end{cases}
\end{equation}
We then define
\begin{equation}
p_D(x):=\sum_{|\alpha|=D} \gamma(\alpha) x^\alpha.
\end{equation}
We claim that $S_D=p_D \ell$ has only non-negative coefficients.
Consider the term in $S_D$ corresponding to the $n$-tuple
$A=(A_1,A_2,\ldots,A_n)$.  The coefficient of this term is
\begin{equation}
c(A):=\sum_{k=1}^n \gamma(A-\bfe^k),
\end{equation}
where we take $\gamma(A-\bfe^k)$ to be 0 if $A_k=0$.  Since all
negative coefficients are equal to $-1$ and all positive
coefficients are equal to $n-1$, to show that $c(A)\geq 0$, it
suffices to show that if there exists $k_1$ such that
$\gamma(A-\bfe^{k_1})=-1$, there exists $k_2 \ne k_1$ such that
$\gamma(A-\bfe^{k_2})=n-1$.

If $\gamma(A-\bfe^{k_1})=-1$, then by our
definition of $\gamma(\alpha)$, $A_{k_1}-1\geq 1$ and, for all $k
\ne k_1$, $A_k \geq 1$.  Thus all $n$ of the numbers
$\gamma(A-\bfe^k)$ are non-zero.  We consider two cases.

In the first case, suppose there exists $k \ne k_1$ such that
$A_k-1=0$.  Then $\gamma(A-\bfe^k)=n-1$ and
$c(A)$ is indeed non-negative.

In the second case, for all $k$, $A_k - 1 >0$. For each $k$, we
consider
\begin{equation}\label{eq: determining gamma}
\sum_{j=1}^{n-1} j(A-\bfe^k)_j-D=
\sum_{j=1}^{n-1}j A_j- D - k.
\end{equation}
Since $k$ ranges over $\{1,2,\ldots,n\}$, the $n$ numbers in
\eqref{eq: determining gamma} are consecutive and thus range over
all congruence classes modulo $n$.  Therefore there exists a $k_2$
for which $\sum j A_j - D - k_2 \equiv 0 \mod n$, so that
$\gamma(A-\bfe^{k_2})=n-1$.  Thus in this case as well, $c(A)$ is non-negative.

Now consider $N_-(p_D)/N_+(p_D)$.  By (ii), this is bounded
above by $n-1$.  Thus (iii) will follow if we show that this ratio
is bounded below by a function of $n$ and $D$ that tends to $n-1$
as $D$ tends to infinity.

Since we will let $D \to \infty$, but $n$ is fixed, we may assume without loss of generality that $D>3n$.  As above, write $S_D(x)=\sum_{|A|=D+1} c(A) x^A$.  Define a subset of multi-indices $A$ of length $D+1$ by
\begin{equation}
I(S_D):=\{\,A: |A| = D+1 \;\text{and}\; A_k \geq 2 \quad \text{for all $k$}\,\}.
\end{equation}
These ``interior" multi-indices are those for which no zero appears in a multi-index associated with a term in $p_D$ contributing to $c(A)$. Thus exactly $n$ non-zero coefficients from $p_D$ contribute to $c(A)$, with precisely $n-1$ of them negative.  Since each negative coefficient in $p_D$ contributes to at most $n$ terms in $S_D$,
\begin{equation}\label{eq: first ineq. involving N-}
nN_-(p_D)\geq (n-1)|I(S_D)|.
\end{equation}
To determine the size of $I(S_D)$, consider the function on $I(S_D)$:
\begin{equation}
g(A):=(A_1-2,A_2-2,\ldots,A_n-2).
\end{equation}
One checks that $g$ is a bijection between $I(S_D)$ and
$\{\,B=(B_1,\ldots,B_n): B_k \geq 0 \;\text{and}\; |B|=D-2n+1\,\}$.
Since $|\{B:B_k \geq 0 \;\text{and}\; |B|=D+1-2n\}|$ equals the number of monomials of degree $D+1-2n$ in $n$ variables,
\begin{equation}
|I(S_D)|=|\{B:B_k\geq 0 \;\text{and}\; |B|=D+1-2n\}|=\binom{D-n}{n-1}.
\end{equation}
Combining with \eqref{eq: first ineq. involving N-} gives
\begin{equation}
N_-(p_D)\geq \frac{n-1}{n} \binom{D-n}{n-1}.
\end{equation}
Since $p_D$ has a non-zero coefficient for every monomial of degree $D$ in $n$ variables,
\begin{equation}
\begin{split}
\frac{N_-(p_D)}{N_+(p_D)} & = \frac{N_-(p_D)}{\binom{D+n-1}{n-1} - N_-(p_D)}\\
& \geq \frac{\frac{n-1}{n}\binom{D-n}{n-1}}{\binom{D+n-1}{n-1} -
\frac{n-1}{n}\binom{D-n}{n-1}}.
\end{split}
\end{equation}
Since $n$ is fixed and we will take a limit as $D \to \infty$, we need only determine the leading-order term in the numerator and the denominator of the last expression.
The numerator is a polynomial in $D$ of degree $n-1$ with leading coefficient $\frac{n-1}{n}\cdot \frac{1}{(n-1)!}$,
whereas the denominator is a polynomial in $D$ of degree $n-1$ with leading coefficient $\frac{1}{(n-1)!}-\frac{n-1}{n}\cdot\frac{1}{(n-1)!}$.  Thus
\begin{equation}
\lim_{D \to \infty}
\frac{\frac{n-1}{n}\binom{D-n}{n-1}}{\binom{D+n-1}{n-1}-\frac{n-1}{n}
\binom{D-n}{n-1}}= \frac{\frac{n-1}{n} \cdot \frac{1}{(n-1)!}}{\frac{1}{(n-1)!} - \frac{n-1}{n}\cdot \frac{1}{(n-1)!}}=n-1.
\end{equation}
This completes the proof of (iii) and of the theorem.
\end{proof}


\section{The general case for $d=1$}

If it were possible to replace an arbitrary $r$ for which
$r(z,\bar{z})\norm{z}^2$ is a squared norm with a diagonal $\tilde{r}$ of
the same signature for which $\tilde{r}(z,\bar{z})\norm{z}^2$ is a squared
norm, the results of the previous section would imply the general results.
Although it appears that such a reduction to the diagonal case is not
possible, we show that it is possible to replace an $r$ as above with an
$\tilde{r}$ with
\begin{equation}
\tilde{r}(z,\bar{z})
=
\left\langle
\left[\begin{smallmatrix} I & 0\\0&-I \end{smallmatrix}\right]
\left[\begin{smallmatrix}\tilde{A}\\ \tilde{B} \end{smallmatrix} \right]\sZ ,
\left[\begin{smallmatrix}\tilde{A}\\ \tilde{B} \end{smallmatrix} \right]\sZ
\right\rangle
\end{equation}
with the same signature as $r$, but with $\left[\begin{smallmatrix}\tilde{A}\\\tilde{B} \end{smallmatrix} \right]$ in a partial row-echelon form.

We first establish an elementary proposition.
\begin{prop} \label{prop:lambdanorm}
If
\begin{equation}
\bigl(
\snorm{f(z)}^2-\snorm{g(z)}^2
\bigr)
\snorm{z}^2
\end{equation}
is a squared norm, then for every $\lambda \in [0,1]$
\begin{equation}
\bigl(
\snorm{f(z)}^2-\lambda\snorm{g(z)}^2
\bigr)
\snorm{z}^2
\end{equation}
is also a squared norm.
\end{prop}
\begin{proof}
For any $\lambda \in [0,1]$,
\begin{equation}
\begin{split}
\left(\snorm{f(z)}^2 - \lambda\snorm{g(z)}^2 \right)\snorm{z}^2 &=
\left(\snorm{f(z)}^2 -\snorm{g(z)}^2 \right)\snorm{z}^2 +
(1-\lambda)\snorm{g(z)}^2\snorm{z}^2 \\
&= \left(\snorm{f(z)}^2 - \snorm{g(z)}^2 \right)\snorm{z}^2 +
\snorm{\sqrt{1-\lambda} \, g \otimes z}^2.
\end{split}
\end{equation}
Since a sum of squared norms is itself a squared norm, the claim
holds.
\end{proof}

The next lemma is of critical importance.
\begin{lemma} \label{lemma:normform}
Suppose $r(z,\bar{z}) = \left\langle
\left[\begin{smallmatrix}I_{N_+} & 0\\
0&-I_{N_-}\end{smallmatrix}\right]
\left[\begin{smallmatrix}A\\
B\end{smallmatrix}\right]
\sZ ,
\left[\begin{smallmatrix}A\\
B\end{smallmatrix}\right] \sZ \right\rangle$ has signature pair
$(N_+,N_-)$ (so that $A$ and $B$ have rank $N_+$ and $N_-$, resp.), and
suppose that $r(z,\bar{z})\norm{z}^2$ is a squared norm.  Then
there exists $\tilde{r}(z,\bar{z}) = \left\langle
\left[\begin{smallmatrix}I_{N_+} & 0\\
0&-I_{N_-}\end{smallmatrix}\right]
\left[\begin{smallmatrix}\tilde{A}\\
\tilde{B}\end{smallmatrix}\right] \sZ ,
\left[\begin{smallmatrix}\tilde{A}\\
\tilde{B}\end{smallmatrix}\right] \sZ \right\rangle$ with the same
signature pair as $r$ such that $\tilde{r}(z,\bar{z})\norm{z}^2$
is also a squared norm and the matrix
$\left[\begin{smallmatrix}\tilde{A}\\
\tilde{B}\end{smallmatrix}\right]$ is in row-echelon form up to
permutation of rows. We will say that such a matrix is in \emph{partial
row-echelon form}.
\end{lemma}

\begin{proof}
For clarity, we suppress the subscripts on our identity matrices and write
$r(z,\bar{z}) = \left\langle
\left[\begin{smallmatrix}I & 0\\
0&-I\end{smallmatrix}\right]
\left[\begin{smallmatrix}A\\
B\end{smallmatrix}\right]
\sZ ,
\left[\begin{smallmatrix}A\\
B\end{smallmatrix}\right]
\sZ \right\rangle$.
Because unitary matrices of the form
$\left[\begin{smallmatrix}U_1 & 0\\
0&U_2\end{smallmatrix}\right]$ (with $U_1$ and $U_2$ unitary and
of dimension $N_+ \times N_+$ and $N_- \times N_-$, resp.) commute with
$\left[\begin{smallmatrix}I & 0\\0&-I
\end{smallmatrix}\right]$, we may write
\begin{equation}
\begin{split}
r(z,\bar{z}) & =
\left\langle
{\left[\begin{smallmatrix}U_1 & 0\\
0&U_2\end{smallmatrix}\right]}^*
\left[\begin{smallmatrix}U_1 & 0\\
0&U_2\end{smallmatrix}\right]
\left[\begin{smallmatrix}I & 0\\
0&-I\end{smallmatrix}\right]
\left[\begin{smallmatrix}A\\
B\end{smallmatrix}\right]
\sZ ,
\left[\begin{smallmatrix}A\\
B\end{smallmatrix}\right]
\sZ \right\rangle
\\
& =
\left\langle
\left[\begin{smallmatrix}U_1 & 0\\
0&U_2\end{smallmatrix}\right]
\left[\begin{smallmatrix}I & 0\\
0&-I\end{smallmatrix}\right]
\left[\begin{smallmatrix}A\\
B\end{smallmatrix}\right]
\sZ ,
\left[\begin{smallmatrix}U_1 & 0\\
0&U_2\end{smallmatrix}\right]
\left[\begin{smallmatrix}A\\
B\end{smallmatrix}\right]
\sZ \right\rangle
\\
& =
\left\langle
\left[\begin{smallmatrix}I & 0\\
0&-I\end{smallmatrix}\right]
\left[\begin{smallmatrix}U_1A\\
U_2 B\end{smallmatrix}\right]
\sZ ,
\left[\begin{smallmatrix}U_1 A\\
U_2 B\end{smallmatrix}\right]
\sZ \right\rangle
.
\end{split}
\end{equation}

By choosing the $U_i$ appropriately,  we put $A$ and $B$ individually into row-echelon form. We do {\it not} achieve
a {\it reduced} row-echelon form. We may not be able to eliminate
non-zero entries above the pivots, and our pivots need not be 1s.
Note that the matrix
$C:=\left[\begin{smallmatrix}A\\B
\end{smallmatrix}\right]$ need not be in row-echelon form, even
after permuting the rows.

What kinds of transformations can we apply to $C$ to reduce it further? Let $T$ be an $(N_+ + N_-)\times
(N_+ + N_-)$ matrix.  Then $\langle I' T C \sZ, T C \sZ \rangle =
\langle I' C \sZ, C\sZ \rangle $ if and only if $T^* I' T=I'$. Consider the leftmost column of $C$.  If it does not have a pivot
of either $A$ or $B$, we set it aside.  If
it has a pivot of $A$ or a pivot of $B$, but not both, we again put
the column aside.  The row containing the pivot may now also be
set aside. If we never reach a column with both a pivot
of $A$ and a pivot of $B$, then $C$ is already in the desired
form.

Suppose, then, that we reach a column containing both a pivot
of $A$ and a pivot of $B$.  Consider the two rows containing the
pivots.  Both contain only zeros to the left of the
pivot.  We represent these two rows by the $2 \times 2$ matrix
\begin{equation}
\begin{bmatrix}a_1 & a_2\\ b_1 & b_2 \end{bmatrix},
\end{equation}
where $a_1$ and $b_1$ are non-zero complex numbers  and $a_2$ and
$b_2$ are row vectors containing the rest of the entries of the
two rows under consideration. Thus in order to find a
transformation $T$ so that $r(z,\bar{z})=\langle I' T C \sZ , T C
\sZ \rangle $ and $T C$ has a single pivot in this column,
appearing in the position formerly occupied by $a_1$, it suffices
to find a $2 \times 2$ matrix $T$ such that $T^*
\left[\begin{smallmatrix}1&0\\0&-1
\end{smallmatrix}\right] T =\left[\begin{smallmatrix}1&0\\0&-1
\end{smallmatrix}\right]$ and $T
\left[\begin{smallmatrix}a_1&a_2\\b_1&b_2
\end{smallmatrix}\right] =
\left[\begin{smallmatrix}a_1'&a_2'\\0&b_2'
\end{smallmatrix}\right]$.

If $T=[t_{ij}]$, the first of these requirements yields
\begin{gather}
\abs{t_{11}}^2-\abs{t_{21}}^2=1\\
\bar{t}_{11}t_{12}-\bar{t}_{21}t_{22}=0\\
\abs{t_{22}}^2-\abs{t_{12}}^2=1.
\end{gather}
In order for the second to be satisfied, we require
\begin{equation}
b_1' = t_{21}a_1+t_{22}b_1= 0.
\end{equation}
Thus we need $t_{21}=-t_{22} \frac{b_1}{a_1}$.  An elementary
calculation shows that $T$ is necessarily of the form
$\left[\begin{smallmatrix}e^{i\theta}t_{22}
& -e^{i\theta}t_{22}\overline{\left(\frac{b_1}{a_1}\right)}\\ -t_{22}\frac{b_1}{a_1} & t_{22}
\end{smallmatrix}\right]$
where $t_{22}$ also satisfies
\begin{equation}
|t_{22}|^2\left(1- \abs{\frac{b_1}{a_1}}^2\right)=1.
\end{equation}
Thus for this $r$, if $\abs{a_1}>\abs{b_1}$, it is possible to replace $C$
with a matrix $C'$ of the same rank in which $a_1'$ is non-zero, but $b_1'=0$.

The only situation left to consider is when $\abs{a_1}\leq \abs{b_1}$.  In
this case we modify $r$.  Since $r(z,\bar{z})=
\norm{A\sZ}^2-\norm{B\sZ}^2$ and $r(z,\bar{z})\snorm{z}^2$ is a
squared norm, by Proposition \ref{prop:lambdanorm}, for any
$\lambda \in [0,1]$, $\left(\norm{A\sZ}^2 -\lambda \norm{B\sZ}^2
\right)\norm{z}^2$ is a squared norm, and
$\tilde{r}(z,\bar{z}):=\norm{A\sZ}^2 -\lambda \norm{B\sZ}^2$ and
$r$ have the same signature pair.  Observe,
\begin{equation}
\begin{split}
\tilde{r}(z,\bar{z})&= \snorm{A\sZ}^2 -
\snorm{\sqrt{\lambda}\,B\sZ}^2\\
&= \left\langle I' \left[\begin{smallmatrix}A\\ \sqrt{\lambda}B
\end{smallmatrix}\right]\sZ, \left[\begin{smallmatrix}A \\ \sqrt{\lambda}B
\end{smallmatrix}\right]\sZ \right\rangle.
\end{split}
\end{equation}
Thus if both $a_1$ and $b_1$ are non-zero, but $|a_1| \leq |b_1|$,
through an appropriate choice of $\lambda$, we can replace $r$
with an $\tilde{r}$ having the same signature pair as $r$ with
matrix $\tilde{C}$ having non-zero entries in precisely the same
positions as in $C$, but with the property that $|a_1|>
\sqrt{\lambda}|b_1|$.  We may thus now apply a transformation $T$
as above to achieve the desired reduction of the matrix
$\tilde{C}$. Continuing in this manner, we eventually obtain an
$\tilde{r}$ with the same signature as the original $r$, but with
the matrix
$\left[\begin{smallmatrix}\tilde{A}\\ \tilde{B}\end{smallmatrix} \right]$
in partial row-echelon
form.
\end{proof}

We can now prove the first part of Theorem~\ref{thm:thmd1}.

\begin{lemma} \label{lemma:thmd1part1}
Let $r(z,\bar{z})$ be a real polynomial on $\C^n$, $n \geq 2$,
and
suppose that $r(z,\bar{z}) \norm{z}^2$ is a squared norm.  Let $(N_+,N_-)$
be the signature pair of $r$.  Then
\begin{equation}
\frac{N_-}{N_+} < n-1.
\end{equation}
\end{lemma}

\begin{proof}
Let $\sZ_k$ denote the vector of all
holomorphic monomials in $n$ variables of degree at most $k$.
Order the monomials as above, and note that multiplication by
$z_j$ preserves the order. In light of Lemma~\ref{lemma:normform},
we may assume $r(z,\bar{z}) = \langle I' C \sZ_k, C \sZ_k \rangle$
where $C$ is in partial row-echelon form and $I'$ is the diagonal matrix
$\left[\begin{smallmatrix}I&0\\0&-I \end{smallmatrix}\right]$ with
signature $(N_+,N_-)$.

Let $C_j$ be the matrix defined by
\begin{equation}
\bigl( C \sZ_k \bigl) z_j = C_j \sZ_{k+1} .
\end{equation}
Because $C$ is in partial row-echelon form, $C_j$ is as well. Then
\begin{equation}
\begin{split}
r(z,\bar{z}) \norm{z}^2 & = \sum_{j=1}^n \abs{z_j}^2 \langle I' C
\sZ_k, C \sZ_k \rangle
\\
& = \sum_{j=1}^n \langle I' C_j \sZ_{k+1}, C_j \sZ_{k+1} \rangle
\\
& =
\left\langle
\begin{bmatrix}
I' & 0 & \cdots & 0 \\
0 & I' & \cdots & 0 \\
\vdots & \vdots & \ddots & \vdots \\
0 & 0 & \cdots & I'
\end{bmatrix}
\begin{bmatrix}
C_1 \\
C_2 \\
\vdots \\
C_n
\end{bmatrix}
\sZ_{k+1}
,
\begin{bmatrix}
C_1 \\
C_2 \\
\vdots \\
C_n
\end{bmatrix}
\sZ_{k+1} \right\rangle = \langle \tilde{I'} \tilde{C} \sZ_{k+1},
\tilde{C} \sZ_{k+1} \rangle .
\end{split}
\end{equation}
The matrix $\tilde{C}$ is not in partial row-echelon form and is not
even of full rank.  It is, however, in a special form that we can
exploit.
Although several rows may have their
leading term in the same column, since each $C_j$ is in partial row-echelon
form, each column can contain the leading terms for at most $n$
rows.

At this point, the precise ordering of the rows of $\tilde{C}$ is
not important; we are only interested in the numbers of rows
associated with positive (resp., negative) entries of $I'$ and the
linear relationships between the two sets.  We thus re-order the rows. Let $P$ be the matrix containing the $nN_+$ rows of
$\tilde{C}$ associated with positive entries in $I'$ and let $Q$
be the matrix consisting of the $nN_-$ rows associated with
negative entries of $I'$.  Since $\tilde{C}^*\tilde{I'}\tilde{C}$
is positive semidefinite, so is
\begin{equation}\label{eq: semidefinite}
\begin{bmatrix} P \\Q \end{bmatrix}^*
\begin{bmatrix}I_{nN_+} & 0 \\ 0 & -I_{nN_-} \end{bmatrix}
\begin{bmatrix}P \\ Q \end{bmatrix} = P^* P - Q^* Q,
\end{equation}
and hence {\em the rows of $Q$ are in the linear span of the rows of $P$}.
Within $P$ and $Q$, we may
assume that, if the leading
term of row $i$ appears in column $j$, then the leading term of
row $i+1$ is either in column $j$ or in some column to the
right of column $j$.

Let $m_+$ denote the number of columns of $P$ containing the
leading term of at least one row of $P$, and let $e_+=nN_+ - m_+$.
We think of $e_+$ as the number of ``extra" rows.  Since we can find $m_+$
rows of $P$ with leading terms in $m_+$ distinct columns,
$\rank(P^t) \geq m_+$ and
$\operatorname{nullity}(P^t) =  nN_+ - \rank(P^t) \leq nN_+ - m_+ = e_+$.

More is true; let $h$ denote the number of columns of $Q$ that contain the
leading term of a row of $Q$, but for which the corresponding
column of $P$ is not one of the $m_+$ counted above.  We thus have
a collection of $m_+ + h$ rows of $\tilde{C}$ that are linearly
independent.  On the other hand,
since all rows of $Q$ are in the linear span of
the rows of $P$, $\rank(P^t) \geq m_+ + h$.  Thus
\begin{equation}
\operatorname{nullity}(P^t) = nN_+ - \rank(P^t) \leq nN_+ - m_+ -h = e_+ - h.
\end{equation}
In particular, $h \leq e_+$.

$Q$ has exactly $nN_-$ rows.  However, by
distinguishing two types of rows of $Q$, we can estimate the
number of rows of $Q$ in terms of the number of rows of $P$. Our
first type of row of $Q$ is one with leading term in one of the
$h$ columns counted above.  Since no row of $P$ has leading term
in such a column, there could be as many as $n$ rows of $Q$ with
leading term in a single such column.  $Q$ therefore has at most
$n h$ such rows.  The second type of row of $Q$ is one with leading
term in one of the $m_+$ columns corresponding to a column of $P$
containing a leading term.  Since one of the at most $n$ rows with
a leading term in this column must be in $P$, $Q$ has at most
$(n-1)m_+$ rows of the second type.

This number $(n-1)m_+$ is
still an overestimate for two reasons. First, of the $m_+$ columns, the
left-most has only a single entry, and it appears in $P$.  To see this,
consider the initial monomials of the $f_j$ and $g_j$.  Let $z^\alpha$ be
the one that comes first in the monomial order.  If it were the initial
monomial of, say, $g_J$, then $z_1 g_J$ would have an initial monomial $z_1
z^\alpha$ coming before the initial monomial of any of the $z_k f_j$,
contradicting the fact that $z_1 g_J$ is in the span of the $z_k f_j$.  Thus
$z^\alpha$ is the initial monomial of one of the $f_j$.  Since all the $f_j$
have distinct initial monomials and because our monomial order is
multiplicative, the left-most column of
$\left[\begin{smallmatrix}P\\Q \end{smallmatrix} \right]$
containing a non-zero entry is that corresponding
to $z_1z^\alpha$, and it contains precisely one non-zero entry.  Second, we
must account for the additional $e_+$ rows in $P$ that
also have leading term in one of the $m_+-1$ columns.  Thus
$(n-1)(m_+-1) - e_+$ is still an upper bound for the number of rows in
$Q$ of this second type.  We find:
\begin{equation}
\begin{split}
nN_- &\leq n h + (n-1)(m_+-1) - e_+\\
&\leq n e_+ + (n-1)m_+ - e_+-(n-1)\\
&= (n-1)(e_+ + m_+)-(n-1)\\
&< (n-1)nN_+ .
\end{split}
\end{equation}
\end{proof}


\section{Upper bound on $N_-/N_+$ for $d > 1$}

\begin{lemma}\label{lemma: N-/N+ for d>1}
Let $r(z,\bar{z})$ be a real polynomial on $\C^n$, $n \geq 2$,
and
suppose that $r(z,\bar{z}) \norm{z}^{2d}$ is a squared norm.  Let $(N_+,N_-)$
be the signature pair of $r$.  Then
\begin{equation}\label{eq: N-/N+ for d>1.}
\frac{N_-}{N_+} < \binom{n-1+d}{d} -1.
\end{equation}
\end{lemma}

\begin{proof}
We follow the proof of Lemma~\ref{lemma:thmd1part1}.
When we multiply $\langle I' C \sZ_k , C \sZ_k \rangle$
by $\norm{z}^{2d}$ instead of $\norm{z}^2$,
we obtain $\binom{n-1+d}{d}$ matrices $C_j$ rather than $n$.  More explicitly,
order the degree $d$ multi-indices and let $\alpha$
be the $j$th multi-index.  Let $C_j$ be the matrix defined by
\begin{equation}
\bigl( C \sZ_k \bigl) z^\alpha = C_j \sZ_{k+d} .
\end{equation}
As above, since
$C$ is in partial row-echelon form and the monomial order is multiplicative, $C_j$ is in partial row-echelon form as well. Then
\begin{equation}
\begin{split}
r(z,\bar{z}) \norm{z}^{2d} & =
\left\langle
\begin{bmatrix}
I' & 0 & \cdots & 0 \\
0 & I' & \cdots & 0 \\
\vdots & \vdots & \ddots & \vdots \\
0 & 0 & \cdots & I'
\end{bmatrix}
\begin{bmatrix}
C_1 \\
C_2 \\
\vdots \\
C_{\binom{n-1+d}{d}}
\end{bmatrix}
\sZ_{k+d}
,
\begin{bmatrix}
C_1 \\
C_2 \\
\vdots \\
C_{\binom{n-1+d}{d}}
\end{bmatrix}
\sZ_{k+d} \right\rangle
\\
& = \langle \tilde{I'} \tilde{C} \sZ_{k+d},
\tilde{C} \sZ_{k+d} \rangle .
\end{split}
\end{equation}
In the matrix
$\tilde{C}$, each column contains the leading term of at most
$\binom{n-1+d}{d}$ rows, though, as above, the left-most non-zero column contains only a single non-zero entry since it comes about by multiplying the least monomial $z^\alpha$ in all the $f_j$ by $z_1^d$.  Thus in a manner identical to the above we obtain:
\begin{equation}
\binom{n-1+d}{d} N_- < \left( \binom{n-1+d}{d} - 1 \right)
\binom{n-1+d}{d} N_+ .
\end{equation}
\end{proof}

\begin{remark}
The proof does not use anything about $\norm{z}^{2d}$ except
that it is a squared norm,
its matrix of coefficients is diagonal, and it has
rank $\binom{n-1+d}{d}$.  Therefore we also obtain the following statement.
\begin{cor}
Let $r(z,\bar{z})$ be a real polynomial on $\C^n$, and consider
$s(z,\bar{z}) = \sum_{j=1}^L \abs{z^{\alpha_j}}^2$, where
$\alpha_1,\dots,\alpha_L$ are distinct multi-indices.
Suppose $r(z,\bar{z}) s(z,\bar{z})$ is a squared norm.  If $(N_+,N_-)$
is the signature pair of $r$, then
\begin{equation}
\frac{N_-}{N_+} < L -1.
\end{equation}
\end{cor}
\end{remark}


\section{A class of examples for $d>1$}

Lemma \ref{lemma: N-/N+ for d>1} merely gives an upper
bound for $N_-/N_+$ for $d>1$; it remains to determine whether the
result is sharp.

We first discuss the case $n=2$.
Lemma \ref{lemma: N-/N+ for d>1} gives $N_-/N_+ < d$, which we claim is sharp for all $d$.  To prove this, we construct a family $\{p_{D}\}$ of polynomials in two real variables such that $p_D(x,y)(x+y)^d$ has all non-negative coefficients and the ratio $N_-(p_D)/N_+(p_D)$ of negative to positive coefficients tends to $d$ as $D \to \infty$. The idea of the construction is quite simple; define $p_D(x,y) =\sum c_j x^{D-j}y^j$ where the first and last coefficients are positive and the interior coefficients repeat a pattern of $d$ negatives followed by a positive.  

More explicitly, suppose $D=(d+1)m$ for $m\in \bN$ and define $p_D(x,y)=\sum_{j=0}^D \gamma(D-j,j) x^{D-j} y^j$, where
\begin{equation}
\gamma(D-j,j)=\begin{cases}2^d -1 & j \equiv 0 \mod d+1 \\
-1 & \text{otherwise} \end{cases}.
\end{equation}
For this family, $N_-(p_D)/N_+(p_D)=dD/(D+d+1)$, which tends to $d$ as $D \to \infty$.  It only remains to verify as we did in the proof of part (iii) of Theorem \ref{thm:monthm1} that the coefficients of $p_D$ have been chosen so that $S_D$ has all nonnegative coefficients. We omit the details.

When $n=3$, Lemma~\ref{lemma: N-/N+ for d>1}
gives
\begin{equation}\label{eq: N-/N+ for n=3, d>1.}
\frac{N_-}{N_+} < \binom{n-1+d}{d} -1=\frac{(d+1)(d+2)}{2}-1.
\end{equation}
When $d=1$, this gives $N_-/N_+ < 2$, which we know to be sharp.
It remains open whether \eqref{eq: N-/N+ for d>1.} is sharp for $d>1$.

\begin{remark}
For $n=3$, we were able to construct
polynomials $p$ for which $S:=p \cdot \ell^d$
has all non-negative coefficients and with
\begin{equation}\label{eq: optimal ratio for knight's moves}
\frac{N_-(p)}{N_+(p)} \geq \left\lfloor \frac{(d+2)^2}{3}
\right\rfloor -1 - \varepsilon.
\end{equation}
We omit the details; we simply mention that
the construction can be done by considering the diagram to be an infinite plane
and by using a pattern of $P$s generated by two generalized
knight moves.  Computer experimentation suggests this bound may, in fact,
be optimal.  Therefore, we suspect \eqref{eq: N-/N+ for n=3, d>1.}
is not sharp.
\end{remark}

Next we find examples that show that the bound
\eqref{eq: N-/N+ for d>1.} is of the right order, i.e.,
for a fixed $n$, of order  $d^{n-1}$.
This lemma is the last part of the proof of Theorem~\ref{thm:thmalld}.

\begin{lemma} \label{lemma:correctorder}
Fix $d>1$ and $n>2$.  There exists a polynomial $p \in
\R[x_1, \dots, x_n]$ for which $S_d(p):=p \cdot \ell^d$ has all
non-negative coefficients and with
\begin{equation}\label{eq:correctorder}
\frac{N_-(p)}{N_+(p)} \geq
{\left(\frac{1}{2^{\frac{n(n-1)}{2}}}\right)} \, d^{n-1} =C(d,n).
\end{equation}
\end{lemma}

\begin{figure}[h!t]
\begin{center}
\begin{minipage}[b]{2.1in}
\centering
\includegraphics{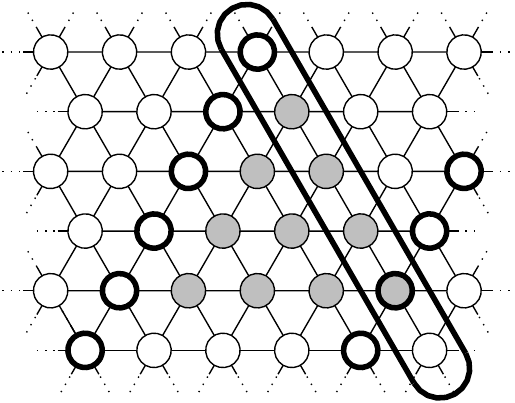}
\end{minipage}
\begin{minipage}[b]{2.1in}
\centering
\includegraphics{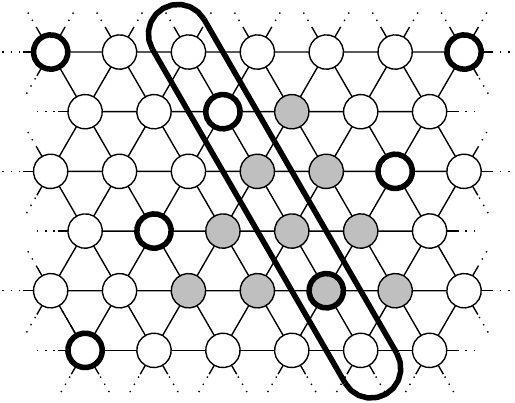}
\end{minipage}
\begin{minipage}[b]{2.1in}
\centering
\includegraphics{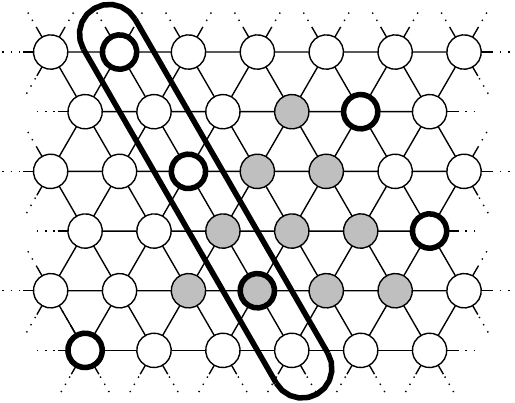}
\end{minipage}
\caption{Diagrams illustrating the first step in the induction for the
proof of Lemma~\ref{lemma:correctorder}. Here, $\nu=0$, $\nu=1$, and
$\nu=2$; and $d=3$.}\label{fig:newtonind}
\end{center}
\end{figure}

\begin{proof}
The proof is by induction on the number of variables $n$.  When $n=2$, we
proved above that we can find polynomials for which $p \cdot \ell^d$ has non-negative coefficients with the ratio $N_-/N_+$ arbitrarily close to $d$. Thus there exists a polynomial for which the ratio exceeds $d/2$.  Thus the result holds for $n=2$.

We proceed to the inductive step.  To simplify notation, we
dehomogenize by setting $x_{n}=1$.  We therefore seek
nonhomogeneous polynomials $p(x_1,\dots,x_{n-1})$ such that
the product $p(x_1,\dots,x_{n-1}) {(x_1+\dots+x_{n-1}+1)}^d$ has nonnegative coefficients.
Suppose that for $n-1$
there exists $p'$ such that
\begin{equation}\label{eq for p'}
\frac{N_-(p')}{N_+(p')} \geq C(d,n-1) - \varepsilon/2.
\end{equation}
That is, $p'$ is a nonhomogeneous polynomial in $(n-1)-1 = n-2$ variables
and multiplying by ${(x_1+\dots+x_{n-2}+1)}^d$ yields a polynomial with
nonnegative coefficients.

Let $x = (x',x_{n-1})$ where $x' \in \R^{n-2}$ so that $p'(x')=\sum_{\alpha} \gamma'(\alpha)x'^\alpha$.  We define
\begin{equation}
p(x)=\sum_{j=0}^k \sum_{\alpha} \gamma(\alpha,j)x'^\alpha x_{n-1}^j
\end{equation}
for appropriately chosen coefficients $\gamma(\alpha,j)$ and for $k$ large.
For each $j$ between 1 and $k-1$, take $\gamma(\alpha,j)=\gamma'(\alpha)$.
In other words, for each of these $j$ we simply repeat the pattern of
positives and negatives from $p'$.  For $j=0$ and $j=k$, we take
sufficiently large positive coefficients to guarantee that $p \ell^d$ has only non-negative coefficients.

When $n=3$, the situation
is illustrated in the first diagram of Figure~\ref{fig:newtonind}.
In the diagram, thick circles are positive
coefficients and thin circles are negative coefficients, as before.
A ``row'' in the diagram corresponding to a fixed power of $x_2$ (a fixed $j$)
is marked with a thick line.  Finally, the shaded circles are the coefficients
that contribute to a single coefficient in $p \cdot \ell^3$. Therefore
any such triangle (or simplex in higher dimensions) must contain a positive
coefficient, as it does in our diagram.
By translating this triangle, we can see the different collections of terms in $p$ that contribute to
different monomials in $S$.
Notice that we cannot place this
triangle any differently so that it includes only negative terms.
Further notice that on the marked ``row'' we have a diagram for
$n=2$. This is how we are using the inductive hypothesis.
The diagram illustrates
only what happens in the ``interior'' and not on the boundary, where $j=0$ or $j=k$.

By taking a large
enough degree to make the contribution to $N_+$ from the ``rows'' $j=0$ and $j=k$
arbitrarily small in the ratio, we find
\begin{equation}
\frac{N_-(p)}{N_+(p)} \geq C(d,n-1) - \varepsilon .
\end{equation}

We can now improve upon this technique; suppose that instead of using $p'$
that satisfied \eqref{eq for p'} for $d$ we take a $p'$ satisfying the equation for $d-1$.
We can then take $\gamma(\alpha, j)=\gamma'(\alpha)$ only for {\it even} $j$ between $1$ and $k$ and can take all $\gamma(\alpha,j)$ for odd $j$ to be negative.
After possibly making the positive
coefficients larger, we conclude that
$p \cdot \ell^d$ has positive coefficients.
This process is illustrated in the second diagram of
Figure~\ref{fig:newtonind}.  Notice that only every second ``row'' contains
positives, and that we took the positives to be closer together
by exactly one on
the rows that do contain positives.

Again by making the degree
large enough we obtain a $p$ such that
\begin{equation}
\frac{N_-(p)}{N_+(p)} \geq 1+ 2C(d-1,n-1) - \varepsilon .
\end{equation}
By repeating this procedure (as illustrated by skipping two ``rows''
in the last diagram of Figure~\ref{fig:newtonind}) we can lower
$d$ by $\nu$ to obtain a $p$ such that
\begin{equation}
\frac{N_-(p)}{N_+(p)} \geq \nu+ (\nu+1)C(d-\nu,n-1) - \varepsilon .
\end{equation}
Picking $\nu = \lfloor \frac{d}{2} \rfloor$ we obtain a polynomial with
\begin{equation}
\frac{N_-(p)}{N_+(p)} \geq \left\lfloor \frac{d}{2} \right\rfloor
+ \left(\left\lfloor \frac{d}{2} \right\rfloor +1\right)
C\left(\left\lceil \frac{d}{2} \right\rceil,n-1\right) - \varepsilon .
\end{equation}
Let us prove $C(d,n) \geq C_n d^{n-1}$ by induction.
For $n=2$, we have seen that we can take $C_{2} = \frac{1}{2}$.
Assume the bound $C(d,n-1) \geq C_{n-1} d^{n-2}$
holds for $n-1$.
We compute for $n > 2$,
\begin{equation}
\left\lfloor \frac{d}{2} \right\rfloor
+ \left(\left\lfloor \frac{d}{2} \right\rfloor +1\right)
C\left(\left\lceil \frac{d}{2} \right\rceil,n-1\right) - \varepsilon
\geq
\left(\frac{d}{2} \right)
C_{n-1} {\left( \frac{d}{2} \right)}^{n-2}
=
\frac{C_{n-1}}{2^{n-1}} d^{n-1} .
\end{equation}
We are allowed to drop the $\varepsilon$ because we are
dropping
$\left\lfloor \frac{d}{2} \right\rfloor$ from the right-hand side.
Therefore we can take $C_{n} =
\frac{C_{n-1}}{2^{n-1}}$ and $C_2 = \frac{1}{2}$ to  obtain
$C_n = \frac{1}{2^{n(n-1)/2 }}$, and therefore \eqref{eq:correctorder}
holds.
\end{proof}

We have finished the proof of Theorem~\ref{thm:thmalld}.  As our final proposition, we show
that the classes $\Psi_j$ are distinct for all $j$.

\begin{prop} \label{prop:qsdistinct}
For all $j = 0,1,2,\dots$,
\begin{equation}
\Psi_j \subsetneq \Psi_{j+1} .
\end{equation}
\end{prop}

\begin{proof}
As above, we need only construct real polynomials.  Define
\begin{equation}
q_k(x) =
x_1^k + x_2^k
+ x_2^{k-1}x_3
+ x_2^{k-1}x_4
+ \dots + x_2^{k-1}x_n
\quad
-
\quad
\varepsilon
x_1x_2^{k-1} .
\end{equation}

The only monomial of
\begin{equation}
-
\varepsilon
x_1x_2^{k-1} (x_1 + x_2 + \dots + x_n)^d
\end{equation}
that does not appear in
\begin{equation}
(x_2^k
+ x_2^{k-1}x_3
+ x_2^{k-1}x_4
+ \dots + x_2^{k-1}x_n)
(x_1 + x_2 + \dots + x_n)^d
\end{equation}
for all $d = 1,2,\dots$,
is the term $-\varepsilon x_1^{d+1}x_2^{k-1}$.  It appears in
$x_1^k (x_1 + x_2 + \dots + x_n)^d$
when $k = d+1$, but not for any smaller $d$.
By taking $\varepsilon > 0$ small enough we obtain that
$q_k \cdot \ell^{d+1}$ has all positive coefficients in this case.

Therefore $q_{d+1}$ is in $\Psi_{d+1}$, but not in $\Psi_d$.  Notice that
$q_{d+1} \notin \Psi_d$ even if we make the negative coefficient arbitrarily
small.
\end{proof}


\def\MR#1{\relax\ifhmode\unskip\spacefactor3000 \space\fi%
  \href{http://www.ams.org/mathscinet-getitem?mr=#1}{MR#1}}

\end{document}